\documentclass[12pt,letterpaper]{amsart}
\usepackage{amsmath, amsthm, amssymb, xspace, mathrsfs}
\usepackage{geometry}
\usepackage[breaklinks=true]{hyperref}
\usepackage{amsmath,amscd}
\usepackage{tikz-cd}

\theoremstyle{plain}
\newtheorem{theorem}{Theorem}[section]

\newtheorem{lemma}{Lemma}[section]
\newtheorem{prop}{Proposition}[section]
\newtheorem{cor}{Corollary}[section]

\theoremstyle{definition}
\newtheorem{defin}{Definition}[section]

\newtheorem*{acknowledgement}{Acknowledgements}

\newcommand{\mf}[1]{\displaystyle{\mathfrak{#1}}}

\newcommand{\comment}[1]{}

\DeclareMathOperator{\spec}{\ensuremath{Spec}}

\DeclareMathOperator{\Gr}{\ensuremath{gr}}

\DeclareMathOperator{\Id}{\ensuremath{Id}}

\DeclareMathOperator{\Sym}{\ensuremath{Sym}}

\DeclareMathOperator{\Aut}{\ensuremath{Aut}}

\begin{document}
\title[Derived invariants of the fixed rings]{Derived invariants of the fixed ring of enveloping algebras of semisimple Lie algebras}
\author{Akaki Tikaradze}
\email{ tikar06@gmail.com}
\address{University of Toledo, Department of Mathematics \& Statistics, 
Toledo, OH 43606, USA}
\begin{abstract}
Let $\mathfrak{g}$ be a semisimple complex Lie algebra, and let $W$
be a finite subgroup of $\mathbb{C}$-algebra automorphisms of the enveloping algebra $U(\mathfrak{g})$.
We show that the derived category of $U(\mathfrak{g})^W$-modules determines isomorphism classes
of both $\mathfrak{g}$ and $W.$ Our proofs are based on the geometry of the Zassenhaus variety of the reduction modulo $p\gg 0$
of $\mathfrak{g}.$ Specifically, we use non-existence of certain \'etale coverings of its smooth locus.

\end{abstract}

\maketitle

\section{Introduction}

Questions regarding finite subgroups of automorphisms of enveloping algebras have been of interest
in ring theory and representation theory for some time now. One such natural question
is as follows. Given a finite subgroup $\Gamma\subset \Aut(U(\mathfrak{g}))$ of automorphisms
of the enveloping algebra of a complex semisimple Lie algebra $\mathfrak{g},$ to what extent can
$\mathfrak{g}$ and $\Gamma$ be recovered from the fixed ring $U(\mathfrak{g})^{\Gamma}?$
One of the earliest results in this direction was obtained by Alev and Polo \cite{AP}. They showed that given a finite subgroup $W$ of  automorphisms of
the enveloping algebra of a semisimple Lie algebra $\mathfrak{g},$ such that the fixed ring $U(\mathfrak{g})^W$ is isomorphic
to an enveloping algebra of a Lie algebra $\mathfrak{g'},$ then $W$ must be trivial and
$\mathfrak{g'}=\mathfrak{g}.$
On the other hand, Caldero \cite{C} showed that given semisimple Lie algebras
$\mathfrak{g}, \mathfrak{g}'$ and finite subgroups of automorphisms
of corresponding enveloping algebras $W\subset \Aut(U(\mathfrak{g})),$
$W'\subset \Aut(U(\mathfrak{g}'))$ such that the corresponding fixed rings $U(\mathfrak{g})^W$ and $U(\mathfrak{g'})^{W'}$ are isomorphic,
then $\mathfrak{g}\cong\mathfrak{g'}.$ If, in addition, $W, W'$ consist of adjoint automorphisms,
 then Caldero also shows that $\mathbb{C}[W]\cong \mathbb{C}[W'].$ 
Moreover, if $W$ is a subgroup of
 $PSL_2$, then  $W\cong W'.$

The following is our main result.

\begin{theorem}\label{main}

Let $\mathfrak{g}, \mathfrak{g'}$ be semisimple complex Lie algebras. Let $W\subset \Aut(U(\mathfrak{g}))$ and $W'\subset \Aut(U(\mathfrak{g'}))$
be finite subgroups of $\mathbb{C}$-algebra automorphisms. 
If the  fixed-point algebras $U(\mathfrak{g})^W$ and $U(\mathfrak{g'})^{W'}$ are derived equivalent,
then $\mathfrak{g}\cong\mathfrak{g'}$ and $W\cong W'.$

\end{theorem}

We also have a similar result about  the fixed-point subalgebras of rings of differential operators on smooth affine varieties.

\begin{theorem}\label{diff}
Let $X, Y$ be smooth affine simply connected varieties over $\mathbb{C}$. Let $W$ and $W'$ be finite subgroups of automorphisms
of $D(X)$ and $D(Y)$ respectively. If the fixed-point algebras $D(X)^{W}$ and $D(Y)^{W'}$ are derived equivalent, then $W\cong W'.$

\end{theorem}

These results and their proof are motivated by the following analogue
for Poisson varieties. Throughout by a $p'$-degree we will mean a degree not divisible by $p.$

\begin{prop}\label{key}
Let  $X$ and $Y$ be affine normal Poisson varieties over  an algebraically closed field $\bold{k}$ of characteristic $p,$
such that their symplectic loci do not admit any nontrivial $p'$-degree  \'etale covering  and have complements of codimension
$\geq 2.$ Let $W$ (resp. $W'$) be a finite subgroup of Poisson automorphisms
of $X$ (respectively $Y$) of order not divisible by $p$. If $X/W\cong Y/W'$ as Poisson varieties,
then there exists an isomorphism of Poisson varieties $f:X\cong Y$ such that
$f_*(W)=W'$, where $f_*:\Aut(X)\to \Aut(Y)$ is the induced isomorphism.

\end{prop}

 Proofs of our main results are based on the reduction modulo a very large prime technique, which
allows a passage to Proposition \ref{key}.

Throughout, given an abelian group $L$, by $L_p$ we denote its reduction modulo $p: L_p=L/pL.$
   We now recall the crucial definition of a Poisson bracket on the center of a reduction
modulo $p$ of an algebra. Given an associative flat $\mathbb{Z}$-algebra $R$ and a prime number $p,$
 then the center $Z(R_p)$ of its reduction modulo $p$  acquires the natural Poisson bracket, which
 we  refer to as the reduction modulo $p$ Poisson bracket,
defined as follows. Given $a, b\in Z(R_p)$, let $z, w\in R$ be their lifts respectively. 
Then the Poisson bracket $\lbrace a, b \rbrace$ is defined to be 
$$\frac{1}{p}[z, w] \mod p\in Z(R_p).$$

 This way we obtain a natural homomorphism from $\Aut(R)$ to the group of Poisson algebra automorphisms
 of $R_p$.


\section{Some results on centers of fixed rings}

In this section we recall a result from \cite{M} and apply it
to our situation. At first, recall the following [\cite{M}, Definition on p. 42].

\begin{defin}
Let $A$ be a Noetherian domain, $Q$ -- its skew field of fractions.
Then  a ring automorphism $f:A\to A$ is said to be an $X$-inner automorphism
if there exists $s\in Q$ such that $f(a)=sas^{-1}$ for all $a\in A$. If $f$ is not $X$-inner, then it
is said to be an $X$-outer automorphism. 
\end{defin}

Clearly in  the above definition, if $A$ is a finite module over its center $Z$, then an $X$-outer automorphism is 
just an outer one. The following lemma is a much weaker version of [\cite{M},  Corollary 6.17], which will
be sufficient for our purposes.

\begin{lemma}\label{M}
Let $A$ be a Noetherian domain, and let $G$ be a finite subgroup of $\Aut(A).$
If all nontrivial elements of $G$ are $X$-outer, then $Z(A^G)=Z(A)^G.$
\end{lemma}

We will use the following  simple corollary of this result. Its proof is essentially identical to [\cite{T}, Proposition 1, the proof of Theorem 1]. 

\begin{cor}\label{action}
Let $\bf{k}$ be a field and let $A$ be a $\bf{k}$-domain
equipped with $\bf{k}$-algebra filtration concentrated in nonnegative degrees, such that its associated graded algebra
$\Gr(A)$ is a commutative domain. Assume that  $A$ is finite over its center.
Let $G\subset \Aut(A)$ be a finite subgroup,
such that $\bf{k}$ contains a primitive $|G|$-th root of unity. Then $Z(A^G)=Z(A)^G$. Moreover,
$G$ acts faithfully on $Z(A).$
\end{cor}
\begin{proof}
In view of Lemma \ref{M}, in order to prove  $Z(A^G)=Z(A)^G,$ it suffices to show that every nonidentity element of $G$ is an
outer automorphism.
Indeed, let $\phi\in G$ be a nontrivial inner automorphism. Let $l$ be the order of $ \phi$, hence $\bf{k}$ contains a primitive 
$l$-th rooth of unity. Let $x\in A, z\in Z(A)$ be such that $\phi(a)=(xz^{-1})a(xz^{-1})^{-1}$ for all $a\in A.$
Thus $\phi(a)=xax^{-1}.$
Since $\phi$ is a nontrivial semisimple automorphism, it has an eigenvalue not equal to 1.
 Let $\xi\neq 1$ be an eigenvalue of $\phi$
with an eigenvector $y\in A, y\neq 0.$ So $xy=\xi yx.$ Hence $\Gr(x)\Gr(y)=\xi \Gr(y)\Gr(x),$ which is a contradiction
since $\Gr(A)$ is a commutative domain.
   
    Now, suppose that $\phi\in \Aut(A)$ is a finite order (order dividing $|G|$) automorphism that acts on $Z(A)$ trivially.
 Let $D$ be the skew field of fractions of $A$ (obtained by inverting nonzero elements of $Z(A)$.) Thus $\phi\in \Aut(D)$
 fixes the center of $D$. Therefore, by the Skolem-Noether theorem, $\phi$ is an inner automorphism of $D$, hence
 an inner automorphism of $A$. Then the above argument shows that $\phi=\Id$. Hence, $G$ acts faithfully
 on $Z(A).$

\end{proof}

\section{Description of  centers of $U(\mathfrak{g}_{\bf{k}}), D(X_{\bf{k}})$}

Let $\mathfrak{g}$ be a complex semisimple Lie algebra, let $G$ be the corresponding
simply connected semisimple algebraic group. Let $\mathfrak{g}_{\mathbb{Z}}, G_{\mathbb{Z}},$
be integral models of $\mathfrak{g}, G,$ respectively.

In this section we recall
some well-known facts and fix the notation about the center of the enveloping algebra 
of $\mathfrak{g}_{\bf{k}}=\mathfrak{g}_{\mathbb{Z}}\otimes\bold{k}$, where 
 $\bf{k}$ is a field of characteristic
$p\gg 0.$ Since we will only be interested in the center of $Z(U(\mathfrak{g}_{\bf{k}}))$ for very large primes
$p$, the choice of an integral model $\mathfrak{g}_{\mathbb{Z}}$ is irrelevant.
     
        Let $l=\text{rank}(\mathfrak{g}),$ and let $f_1,\cdots, f_l\in Z(U(\mathfrak{g}_{\mathbb{Z}}))$ be central elements
        that generate the center of $U(\mathfrak{g}).$ Given a field 
       $\bf{k}$ of characteristic $p\gg 0$, we will denote by $\bar{f}_i, 1\leq i\leq l$ the image of $f_i$
       under the base change homomorphism $U(\mathfrak{g}_{\mathbb{Z}})\to U(\mathfrak{g}_{\bold{k}}).$
Put $\mathfrak{g}_p=\mathfrak{g}_{\mathbb{Z}}/p\mathfrak{g}_{\mathbb{Z}}.$
 Recall that the $p$-center of $U(\mathfrak{g}_p)$, to be denoted by $Z_p(\mathfrak{g}_p),$
    is generated by
  elements of the form $x^p-x^{[p]}, x\in \mathfrak{g}_p.$ It is well-known that we have an isomorphism
  $\Sym(\mathfrak{g}_p)\cong Z_p(\mathfrak{g}_p)$ of $\mathbb{F}_p$-algebras given by $x\to x^p-x^{[p]}, x\in\mathfrak{g}_p.$
Now recall that the reduction modulo $p$ Poisson bracket on $Z(U(\mathfrak{g}_p))$ restricts  
     on $Z_p(\mathfrak{g}_p)$
to the negative of the Kirillov-Kostant bracket \cite{KR}
$$\lbrace a^p-a^{[p]}, b^p-b^{[p]}\rbrace=-([a, b]^p-[a,b]^{[p]}),\quad a\in \mathfrak{g}_p, b\in \mathfrak{g}_p.$$


Let $\bf{k}$ be  an algebraically closed field of characteristic $p.$ Thus 
$Z(U(\mathfrak{g}_{\bold{k}}))=Z(U(\mathfrak{g}_p))\otimes\bold{k}$ can be equipped with the corresponding
$\bf{k}$-linear Poisson bracket.
Denote by $Z_0(\mathfrak{g}_{\bold{k}})$ the image of $Z(U(\mathfrak{g}_{\mathbb{Z}}))$ in 
  $Z(U(\mathfrak{g}_{\bold{k}}))$ (the Harish-Chandra part of the center). So $Z_0(\mathfrak{g}_{\bold{k}})=\bold{k}[\bar{f_1},\cdots, \bar{f_l}].$
Clearly, $Z_0(\mathfrak{g}_{\bold{k}})$
lies in the Poisson center of $Z(U(\mathfrak{g}_{\bold{k}})).$

 Let  $\chi:Z_0(\mathfrak{g}_{\bold{k}})\to \bold{k}$ be a character. 
Then the quotient 
$$Z_{\chi}=Z(U(\mathfrak{g}_{\bf{k}})/\ker{\chi})=Z(U(\mathfrak{g}_{\bold{k}}))/\ker(\chi)$$
is equipped with the induced Poisson bracket.

Next we  recall a well-known theorem of Veldkamp  (see for example [\cite{Ta} Theorem 1.6] or [\cite{MR} Cor.3]) describing the center of $U(\mathfrak{g}_{\bf{k}}).$

\begin{theorem}\label{Veldkamp}
$Z(U(\mathfrak{g}_{\bold{k}}))$ is a free $Z_p(\mathfrak{g}_{\bold{k}})$-module with a basis $\lbrace \bar{f_1}^{a_1}\cdots \bar{f_l}^{a_l}, 0\leq a_i<p\rbrace,$
and $U(\mathfrak{g}_{\bold{k}})^{G_{\bf{k}}}=Z_0(\mathfrak{g}_{\bf{k}}).$
Moreover, we have an isomorphism induced by the multiplication map 
$$Z_p(\mathfrak{g}_{\bold{k}})\otimes_{Z_p(\mathfrak{g}_{\bold{k}})^{G_{\bf{k}}}}U(\mathfrak{g}_{\bold{k}})^{G_{\bf{k}}}\to Z(U(\mathfrak{g}_{\bold{k}})).$$

\end{theorem}

In particular, the above description of  $Z(U(\mathfrak{g}_{\bf{k}}))$ implies that $\spec Z_{\chi}$ is isomorphic as a Poisson variety
to $\mu^{-1}(\chi')$, where 
$\mu:\mathfrak{g}_{\bf{k}}^*\to \spec(Z_0(\mf{g}_{\bf{k}}))\cong\mathfrak{g}_{\bf{k}}^*//G_{\bf{k}}$ is the usual map and $\chi'\in \mathfrak{g}_{\bf{k}}^*//G_{\bf{k}}$
(we do not need to know a precise formula for $\chi'$ here). Therefore, the symplectic locus of $\spec Z_{\chi}$ has a 
complement of codimension at least 2.

  Now let $S\subset\mathbb{C}$ be a finitely generated ring, and let $X$ be a smooth affine variety $X$ over $S.$ 
  Then the center of the reduction  modulo $p$ of its ring
  of (crystalline) differential operators $D(X_p)=D(X)/pD(X)$ is isomorphic to the Frobenius twist of the ring of regular functions
  on the cotangent bundle of $X_p$ (see \cite{BMR}). Moreover, the reduction
  modulo $p$ Poisson bracket on $Z(D(X_p))$ equals to the negative of the usual
  Poisson bracket of the cotangent bundle $T^*(X_p)$. In particular, given a base change $S\to\bf{k}$ to an algebraically closed
  field of characteristic $p,$ then under the induces $\bf{k}$-linear Poisson bracket $\spec Z(D(X)\otimes_S\bold{k})$
  is a symplectic variety.

\section{Proofs}

At first, recall the following well-known result  from algebraic geometry about purity of the branched locus [\cite{SGA} Corollaire 3.3. ].

\begin{theorem}\label{etale}
Let $X$ be a regular connected Noetherian scheme over an algebraically closed field $\bold{k},$ let $U\subset X$ be a nonempty connected open subset.
Then the corresponding map of the \'etale fundamental groups $\pi_1(U)\to \pi_1(X)$ is surjective, and
it is an isomorphism if $X\setminus U$ has codimension $\geq 2$.
\end{theorem}

We also need  the following simple result. Its proof is included
for the reader's convenience. 

\begin{lemma}\label{elementary}
Let $A, B$ be Poisson domains over an algebraically close field $\bf{k}$ of characteristic $p$.
Let $A_1$ be a Poisson $\bf{k}$-subalgebra of $A$. Let $f:A\to B$ be a  $\bf{k}$-algebra isomorphism, such that
$f|_{A_1}$ preserves the Poisson bracket. If $[Frac(A): Frac(A_1)]<p$, then $f$ preserves the Poisson
bracket.
\end{lemma}
\begin{proof}
We may assume that $A, B, A_1$ are fields.
Let $x\in A$. 
Let $d<p$ be the degree of $x$ over $A_1$.
Hence $\sum_{n=0}^d a_nx^n=0$ with $a_d\neq 0$ for some $a_i\in A_1.$
Let $D: A\to A$ be a derivation. Then 
$$D(x)(\sum na_nx^{n-1})=-\sum(D(a_n)x^n).$$
Thus $D$ is determined by $D|_{A_1}$. This implies our assertion. 
\end{proof}

\begin{proof}[Proof of Proposition \ref{key}]
Put $Z=X/W\cong Y/W'.$ Denote by $p_1:X\to Z$ and $p_2:Y\to Z$ the corresponding
quotient maps.
Let $U$ (respectively  $U'$) be the symplectic locus of $X$ (resp. $Y$.)
Let $U_1$ (respectively $U'_1$) be the locus of points in $U$ (resp. $U'$) on which $\Gamma$ (resp. $W$)
acts freely.
Now it is immediate that $U\setminus U_1$ (respectively  $U'\setminus U'_1$) has at least codimension
2 in $U$ (resp. $U'$).  Put $V=p_1(U_1)\cap p_2(U_2).$ 
Then $Z\setminus V$ has codimension at least 2 in $Z.$
Thus $p_1^{-1}(V)$ (resp. $p_2^{-1}(V)$) has complement in $U$ of codimension
at least 2 (resp. complement in $U'$). Hence by
Lemma \ref{etale} $p_1^{-1}(V)$ and  $p_1^{-1}(V)$ do not admit any nontrivial
$p'$-degree \'etale coverings. On the other hand, $p_1:p_1^{-1}(V)\to V$ and $p_2:p_2^{-1}(V)\to V$
are $W$ (respectively $W'$)-Galois covering. Hence
 there exists an
isomorphism $f:p_1^{-1}(V)\to p_2^{-1}(V)$
interchanging actions of $W$ and $W': f_{*}(W)=W'.$
By Lemma \ref{elementary} $f$ preserves the Poisson bracket.
Now since $X\setminus p_1^{-1}(V)$ has codimension at least 2 and $X$ is a normal variety, we conclude that
$\mathcal{O}(p_1^{-1}(V))=\mathcal{O}(X).$ Similarly, $\mathcal{O}(p_2^{-1}(V))=\mathcal{O}(Y).$
Thus, we get the desired compatible isomorphisms $X\cong Y, W\cong W'.$
\end{proof}

Now we can easily prove Theorem \ref{diff}.

\begin{proof}[Proof of Theorem \ref{diff}]
Put $A=D(X)^W, B= D(Y)^{W'}.$ We may chose large enough finitely generated
subring $S\subset\mathbb{C}$, over which $A, B$ are defined, 
such that $A$ and $B$ are derived equivalent over $S.$
Now the standard argument about derived invariance of the Hochschild cohomology
yields that $Z(A_p)\cong Z(B_p)$ as $S_p$-Poisson algebras (see [\cite{T} Lemma 4]) .
On the other hand, using Corollary \ref{action}
for a base change $S\to\bold{k}$
to an algebraically closed field $\bold{k}$ of characteristic $p\gg 0$, we have
$Z(A_{\bold{k}})=Z(D(X_{\bold{k}}))^W$ and $Z(B_{\bold{k}})=Z(D(X_{\bold{k}}))^W .$
Therefore, we have an isomorphism of Poisson $\bold{k}$-algebras
$$Z(D(X_{\bold{k}}))^W\cong Z(D(Y_{\bold{k}})))^{W'}.$$ 
But since $Z(D(X_{\bold{k}}))$ (respectively $Z(D(Y_{\bold{k}}))$) is isomorphic to (the Frobenius twist) of the cotangent 
$T^*(X_{\bold{k}})$ (resp. $T^*(Y_{\bold{k}})$),  we have an isomorphism of Poisson $\bold{k}$-varieties
$$T^*(X_{\bold{k}})/W\cong T^*(Y_{\bold{k}})/W'.$$
Since by the assumption $T^*(X)$ and $T^*(Y)$ are simply connected,  it follows that $T^*(X_{\bold{k}})$ (similarly $T^*(Y_{\bold{k}})$)
  admits no nontrivial $p'$-\'etale covering (see [\cite{T2}, Lemma 5].)
Now Proposition \ref{key} applied to $T^*(X_{\bold{k}})$ and $T^*(Y_{\bold{k}})$ yields the desired isomorphism $W\cong W'$.

\end{proof}

In order to prove Theorem \ref{M} we  need few more lemmas.
In what follows $\mathfrak{g}$ is a fixed complex semisimple Lie algebra
with an integral model $\mathfrak{g}_{\mathbb{Z}}.$ As usual, given a ring $S$
we put $\mathfrak{g}_{S}=\mathfrak{g}_{\mathbb{Z}}\otimes S.$
 Throughout we are using notations from Section 3.

\begin{lemma}\label{faithful}

Let $S\subset\mathbb{C}$ be a finitely generated ring and let
 $\Gamma\subset \Aut(U(\mathfrak{g}_S))$ be a finite subgroup of $S$-automorphisms. 
Suppose that $S$ contains all $|\Gamma|$-th roots of unity.
Then there exists $0\neq f\in S$, such that 
 for any base change to an algebraically closed field $S[f^{-1}]\to\bold{k}$
of characteristic $p\gg 0,$
 if $\chi:Z_0(\mathfrak{g}_{\bf{k}})\to \bf{k}$ is a $\Gamma$-invariant character, then  the action of $\Gamma$ on 
 $Z_{\chi}=Z(U(\mathfrak{g}_{\bold{k}}))/\ker{\chi})$ is faithful.

\end{lemma}
\begin{proof}
There exists a nonzero element $f\in S$, such that for any base change $S[f^{-1}]\to\bf{k}$ the
induces action of $\Gamma$ on $U(\mf{g}_{\bf{k}})$ is faithful.
So $\Gamma\subset \Aut(U(\mathfrak{g}_{\bold{k}})).$
Put $B=U(\mathfrak{g}_{\bf{k}})/\ker(\chi)U(\mathfrak{g}_{\bf{k}}).$ 
Then a proof identical to  [\cite{T}, Proposition 1] shows that the restriction of  the action of $\Gamma$
on  $B$ is faithful.
 So $\Gamma\subset \Aut(B).$
Now by Lemma \ref{action} $\Gamma$ acts faithfully on $Z_{\chi}.$

\end{proof}

The next result plays a crucial role in proving Theorem \ref{M}.

\begin{lemma}\label{connected}

Let $\bf{k}$ be an algebraically closed field of characteristic $p\gg 0.$
Let $X=\spec Z(U(\mathfrak{g}_{\bf{k}}))$ be the Zassenhaus variety of $\mathfrak{g}_{\bf{k}}$.
Let $U$ be the smooth locus of $X.$ Then $U$ does not admit any nontrivial \'etale $p'$-degree covering.
\end{lemma}
\begin{proof}
As explicitly constructed in [\cite{Ta}, Remark 2.4], there exists a morphism of varieties $\phi: \mathfrak{g}^*_{\bold{k}}\to X,$
such that it induces  an isomorphism $\phi^{-1}(U_{rss})\cong U_{rss}$ on an open subset of regular semisimple
elements $U_{rss}\subset U.$ Put  $O=\phi^{-1}(U_{rss}).$ Thus $\phi|_{O}:O\cong U_{rss}.$
Let $W=\phi^{-1}(U).$ Hence the complement of $W$ in $\mathfrak{g}^*_{\bold{k}}$ has codimension at least 2. 
In particular, using Lemma \ref{etale} $W$ admits no nontrivial $p'$-degree \'etale covering.
Let $\pi:Y\to U$ be a $p'$-degree \'etale covering. Let $\pi':Y'\to W$ be its pull-back
via $\phi.$ Therefore, $\pi'$ must be a trivial covering, hence so is its restriction
on $O.$ Thus the restriction of $\pi$ on $U_{rss}$ is trivial, implying the triviality
of the covering $\pi$ (again by Lemma \ref{etale}.)

\end{proof}

\begin{lemma}\label{codim}
Let $S\subset\mathbb{C}$ be a finitely generated ring and let
 $W\subset \Aut(U(\mathfrak{g}_S))$ be a finite subgroup of automorphisms. 
Then there exists $0\neq f\in S$, such that 
 for any base change to an algebraically closed field $S[f^{-1}]\to\bold{k}$
of characteristic $p\gg 0,$  the locus of points in $\spec Z(U(\mathfrak{g}_{\bold{k}}))$ with a nontrivial stabilizer in $W$
 has at least codimension $\geq 2.$ 

\end{lemma}
\begin{proof}
Put $X= \spec Z(U(\mathfrak{g}_{\bold{k}})).$
Assume that there exists a non-identity element $\sigma\in W$, such that $X^g$ has codimension 1 in $X.$
 Put $\Gamma=\langle \sigma\rangle.$ 
Let $\chi\in (\spec Z_0(\mathfrak{g}_{\bold{k}}))^{\Gamma}$ be in the image of $X^{\Gamma}=X^{\sigma}$ under the map $X\to\spec Z_0(\mathfrak{g}_{\bold{k}}).$
Then $\Gamma$ acts on the the quotient $U_{\chi}=U(\mathfrak{g})/\ker(\chi)U(\mathfrak{g}).$ 
Put $Z(U_{\chi})=Z_{\chi}, X_{\chi}=\spec Z(U_{\chi})$ and $Y=X_{\chi}^{\Gamma}.$
We may (and will) view $X_{\chi}$ as a $\Gamma$-stable subvariety of $X$ 
 By Lemma \ref{faithful}, $\Gamma$ acts faithfully on $X_{\chi}.$ So, $Y=X^{\Gamma}\cap X_{\chi}$ has codimension 1 in $X_{\chi}.$ 
But this is a contradiction, since $X_{\chi}$ is a symplectic variety
outside a codimension 2 subset and $\Gamma$ acts faithfully on it preserving the symplectic structure.

\end{proof}

\begin{proof}[Proof of Theorem \ref{main}]
Just as in the proof of Theorem \ref{diff}, we may pick large enough finitely
generated ring $S\subset \mathbb{C}$ over which $W, W'$ are defined, such that
$S$-algebras $U(\mathfrak{g})^W$ and $U(\mathfrak{g'})^{W'}$ are derived equivalent.
Therefore,  after a base change $S\to\bf{k}$ to an algebraically closed field of characteristic $p\gg 0$,
we get a Poisson $\bf{k}$-algebra isomorphism (similarly to the Proof of Theorem \ref{diff})
$$Z(U(\mathfrak{g}_{\bold{k}}))^W\cong Z(U(\mathfrak{g}_{\bold{k}}))^W.$$
Put $X=\spec Z(U(\mathfrak{g}_{\bold{k}})), Y=\spec Z(U(\mathfrak{g'}_{\bold{k}})).$
Then by Lemma  \ref{codim} the locus of points in $X$ (respectively $Y$) with a non-trivial stabilizer in $W$ (resp. $W'$)
 has codimension at least $2.$ Since the smooth loci of $X$ and $Y$ do not admit any nonytivial $p'$-degree \'etale coverings by Lemma \ref{connected},
we may adapt the proof of Proposition \ref{key} to this setting.  Hence we get an isomorphism of  Poisson $\bold{k}$-algebras

$$f:Z(U(\mathfrak{g}_{\bold{k}}))\to Z(U(\mathfrak{g}_{\bold{k}}')),$$
that  interchanges the  actions of $W$ and $W'.$
Now let $\mathfrak{m}$ be a maximal Poisson ideal in $Z(U(\mathfrak{g}_{\bold{k}}))$, and put $\mathfrak{m}'=f(\mathfrak{m}).$
Then we get an isomorphism of Lie algebras $\mathfrak{m}/\mathfrak{m}^2\cong \mathfrak{m'}/\mathfrak{m'}^2.$
It follows easily from the description of $Z(U(\mathfrak{g}_{\bold{k}})$) that $\mathfrak{m}/\mathfrak{m}^2$ (respectively $\mathfrak{g'}_{\bold{k}})$) is isomorphic to a direct sum of  $\mathfrak{g}_{\bold{k}}$ (resp. $\mathfrak{g'}_{\bold{k}}$) with an abelian
Lie algebra (see [\cite{T} Lemma 3].)
This easily yields an isomorphism $\mathfrak{g}_{\bold{k}}\cong \mathfrak{g'}_{\bold{k}}.$
So $\mathfrak{g}\cong \mathfrak{g'}.$

\end{proof}

\begin{acknowledgement} 
I am grateful to R.Tange for several helpful comments. 
I would also like to thank the anonymous referee for many helpful suggestions
that led to improvement of the paper.

\end{acknowledgement}


\end{document}